\documentclass[11pt]{article}
\usepackage{graphicx,amsmath,amsthm,latexsym,amssymb,amstext}
\newtheorem{theorem}{Theorem}[section]
\newtheorem{definition}[theorem]{Definition}

\newtheorem{proposition}[theorem]{Proposition}

\newtheorem{remark}[theorem]{Remark}
\newtheorem{example}{Example}
\numberwithin{equation}{section}
\usepackage[body={14.5cm,22cm}, left=2cm, right=2cm, top=3cm]{geometry}
\usepackage[pdfstartview=FitH,
bookmarksnumbered=true,bookmarksopen=true,%
colorlinks=true,pdfborder=001,citecolor=blue,%
linkcolor=blue,urlcolor=blue]{hyperref}

\begin{document}
\title{\Large \bf Anticipating backward stochastic Volterra integral equations \footnotemark[1]}

\author{Jiaqiang Wen\footnotemark[2] and Yufeng Shi\footnotemark[2] \footnotemark[3]}

\date{}

\renewcommand{\thefootnote}{\fnsymbol{footnote}}

\footnotetext[1]{This work was supported by National Natural Science Foundation of China (Grant Nos. 11371226 and 11231005),
    Foundation for Innovative Research Groups of National Natural Science Foundation of China (Grant  No. 11221061),
    the 111 Project (Grant No. B12023).}

\footnotetext[2]{Institute for Financial Studies and School of Mathematics, Shandong University, Jinan 250199, China.}

\footnotetext[3]{Corresponding author. E-mail addresses: jqwen@mail.sdu.edu.cn, yfshi@sdu.edu.cn}

\maketitle
\begin{abstract}
We introduce and study a new type of integral equations called anticipating backward stochastic Volterra integral equations (anticipating BSVIEs).
In these equations the generator involves not only the present values but also the future values of the solutions.
We obtain the existence and uniqueness theorem and a comparison theorem  for the solutions to these anticipating BSVIEs.
\end{abstract}

\textbf{keywords}: Anticipating backward stochastic Volterra integral equation, Backward stochastic Volterra integral equation, Comparison theorem.

\text{2010 MSC}: 60H10, 60H20.

\section{Introduction}
Stochastic Volterra integral equations (SVIEs, for short) were introduced by Berger and Mizel \cite{Berger},
and developed to the anticipating SVIEs by Pardoux and Protter \cite{Protter},  Al\`{o}s and Nualart \cite{Nualart}.
General backward stochastic Volterra integral equations (BSVIEs, for short) were introduced by Yong \cite{Yong2,Yong3}.
In more details, let $(\Omega,\mathcal{F},P,\mathcal{F}_{t},t\geq 0)$ be a complete stochastic basis such that $\mathcal{F}_{0}$ contains all $P$-null elements of $\mathcal{F}$ and suppose that the filtration is generated by a $d$-dimensional standard Brownian motion $W=\{W(t); t\geq 0\}$.
Let $(Y(\cdot),Z(\cdot,\cdot))$ be the solution of the following backward stochastic Volterra integral equation:
\begin{equation}\label{3}
 Y(t)=\psi(t) + \int_t^T g(t,s,Y(s),Z(t,s),Z(s,t)) ds  - \int_t^T Z(t,s) dW(s), \ \ t\in [0,T],
\end{equation}
where $g:\Omega\times \Delta\times \mathbb{R}^{m}\times \mathbb{R}^{m\times d}\times \mathbb{R}^{m\times d}\rightarrow \mathbb{R}^{m}$
and $\psi:\Omega\times[0,T]\rightarrow \mathbb{R}^{m}$ are given maps with $\Delta=\{ (t,s)\in[0,T]^{2}| \ t\leq s\}$.
Such an equation was introduced by Yong \cite{Yong2,Yong3}.
A special case of (\ref{3}) with $g(\cdot)$ independent of $Z(s,t)$ and $\psi(t)\equiv \xi$ was studied by Lin \cite{Lin} a little earlier.
Some recent developments of BSVIEs can be found in Eduard and Ludger \cite{Kromer}, Shi, Wen and Yong \cite{Wen}, Shi and Wang \cite{Wang},
 Shi, Wang and Yong \cite{Shi2,Shi4}, Wang and Yong \cite{Yong}, Yong \cite{Yong4}, Zhang \cite{Zhangx}, etc., among theories and applications.
The same as anticipating SVIEs, it is a natural question if there are the corresponding ``anticipating'' BSVIEs.


Recently, Peng and Yang \cite{Yang} introduced the anticipating (or anticipated) backward stochastic differential equation (BSDE, for short) as follows,
\begin{equation}\label{21}
  \begin{cases}
    -dY_{t}=f(t,Y_{t},Z_{t},Y_{t+\delta_{t}},Z_{t+\zeta_{t}}) dt - Z_{t} dW_{t}, \ \ \ 0\leq t\leq T; \\
      Y_{t}=\xi_{t}, \ \  Z_{t}=\eta_{t}, \ \ \  T\leq t\leq T+K,
  \end{cases}
\end{equation}
where $\xi_{\cdot}, \eta_{\cdot}$ are given adapted stochastic processes, and $\delta_{\cdot}, \zeta_{\cdot}$ are given nonnegative deterministic functions.
See Pardoux and Peng \cite{Peng}, El Karoui, Peng, and Quenez \cite{Peng2}, Ma and Yong \cite{Ma}, Chen and Wu \cite{Wu}, Yang and Elliott \cite{Yang2}, etc.,
for systematic discussions about BSDEs and anticipating BSDEs.

These tempt us to introduce the following new type of BSVIEs:
\begin{equation}\label{18}
  \begin{cases}
  Y(t)=\psi(t) + \int_t^T g(t,s,Y(s),Z(t,s),Z(s,t),Y(s+\delta_{s}),Z(t,s+\zeta_{s}),Z(s+\zeta_{s},t)) ds \\
      \ \ \ \ \ \ \  \ \ \ - \int_t^T Z(t,s) dW(s), \ \ t\in [0,T]; \\
  Y(t)=\psi(t), \ \ t\in [T,T+K];\\
Z(t,s)=\eta(t,s), \ \ (t,s)\in [0,T+K]^{2}\setminus [0,T]^{2}.
  \end{cases}
\end{equation}
See Section 3 for detailed discussions. We call equation (\ref{18}) the anticipating backward stochastic Volterra integral equation
(ABSVIE, for short).
One can note that, comparing with BSVIE (\ref{3}), the distinct development of ABSVIE (\ref{18}) is that
the generator of (\ref{18}) involves not only the  present values of solutions but also the future ones of solutions.

In this paper, we establish the existence and uniqueness of solutions of ABSVIE (\ref{18}) under Lipschitz condition.
The method used to prove the existence and uniqueness theorem (Theorem \ref{0} below) is convenient than the four steps method in Yong \cite{Yong3}.
Since the comparison theorem is a fundamental tool, which plays an important role in the theory and applications of BSVIEs,
we also prove a comparison theorem for ABSVIEs, which generalises one of the main results in Wang and Yong \cite{Yong}.
Similar to BSVIE (\ref{3}) and the anticipating BSDE (\ref{21}), ABSVIE (\ref{18}) can also be applied in mathematical finance, risk management, especially in the field of stochastic optimal controls.  About this topic, we will give some further studies in the coming future researches.

The rest of the paper is organized as follows. In Section 2, we introduce some preliminaries.
Section 3 is devoted to the proof of the existence and uniqueness theorem for ABSVIEs.
A comparison theorem for ABSVIEs is also established in Section 4.

\section{Preliminaries}

Let $(\Omega,\mathcal{F},P,\mathcal{F}_{t},t\geq 0)$ be a complete stochastic basis such that $\mathcal{F}_{0}$ contains all $P$-null elements of $\mathcal{F}$ and suppose that the filtration is generated by a $d$-dimensional standard Brownian motion $W=\{W(t),t\geq 0\}$.
The Euclidean norm of a vector $x\in\mathbb{R}^{m}$ will be denoted by $|x|$,
and for a $m\times d$ matrix $A$, we define $\| A \|=\sqrt{Tr AA^{\ast}}$.
Given $T>0$, and let $K\geq 0$ be a constant, denote
\begin{equation*}
  \Delta=\{ (t,s)\in[0,T]^{2}| \ 0\leq t\leq s\leq T \}; \ \ \ \widetilde{\Delta}=\{ (t,s)\in[0,T+K]^{2}| \ 0\leq t\leq s\leq T+K \}.
\end{equation*}
Also, for $H=\mathbb{R}^{m},\mathbb{R}^{m\times d}$ and $t\in[0,T],$ denote\\
$\bullet$ $L^{2}(\mathcal{F}_{t};H) =\{\xi:\Omega\rightarrow H \mid \xi$ is $\mathcal{F}_{t}$-measurable, $E[|\xi|^{2}]< \infty\}$;\\
$\bullet$ $L_{\mathcal{F}_{T}}^2(0,T;H)
  =\big \{\psi:\Omega\times [0,T]\rightarrow H \mid \psi(t)$ is $\mathcal{F}_{T\vee t}$-measurable,
 $E\int_0^{T} |\psi(t)|^{2} dt < \infty\big \}$;\\
$\bullet$ $L_{\mathcal{F}}^2(0,T;H)
  =\big \{X:\Omega\times [0,T]\rightarrow H \mid X(t)$ is $\mathcal{F}_{t}$-measurable,
  $E\int_0^T |X(t)|^{2} dt$ $< \infty\big \}$;\\
$\bullet$ $L_{\mathcal{F}}^2(\Delta;H)
  =\big\{Z:\Omega\times \Delta \rightarrow H  \mid Z(t,s)$ is $\mathcal{F}_{s}$-measurable,
        $E\int_0^{T} \int_t^{T} |Z(t,s)|^{2} dsdt$ $< \infty\big\}$;\\
$\bullet$ $L_{\mathcal{F}}^2([0,T]^{2};H)
  =\big\{Z:\Omega\times [0,T]^{2}\rightarrow H  \mid Z(t,s)$ is $\mathcal{F}_{s}$-measurable,
        $E\int_0^{T} \int_0^{T} |Z(t,s)|^{2} dsdt$ $< \infty\big\}$.\\
For any $\beta\geq 0$, let $\mathcal{H}^{2}_{\Delta}$ be the space of all pairs
 $(Y,Z)\in L_{\mathcal{F}}^2(0,T;\mathbb{R}^{m})\times L_{\mathcal{F}}^2(\Delta;\mathbb{R}^{m\times d})$ under the following norm
\begin{equation*}
  \| (Y(\cdot),Z(\cdot,\cdot)) \|_{\mathcal{H}_{\Delta}^{2}}
  \equiv \left[ E\int_0^{T} \left(e^{\beta t}|Y(t)|^{2} + \int_t^{T} e^{\beta s}\|Z(t,s)\|^{2} ds\right) dt \right]^{\frac{1}{2}}< \infty.
\end{equation*}
Clearly, $\mathcal{H}_{\Delta}^{2}$ is a Hilbert space.
Similarly, we can define $L_{\mathcal{F}_{T}}^2(0,T+K;H)$, $L_{\mathcal{F}}^2(0,T+K;H)$,
 $L_{\mathcal{F}}^2(\widetilde{\Delta};H)$, $L_{\mathcal{F}}^2([0,T+K]^{2};H)$ and $\mathcal{H}_{\widetilde{\Delta}}^{2}$.
From the definition, we note that the space $L_{\mathcal{F}_{T}}^2(T,T+K;H)$ is equivalent to the space $L_{\mathcal{F}}^2(T,T+K;H)$.

Let's consider the following BSVIE, which was introduced by Yong \cite{Yong2,Yong3},
\begin{equation}\label{22}
 Y(t)=\psi(t) + \int_t^T g(t,s,Y(s),Z(t,s),Z(s,t)) ds  - \int_t^T Z(t,s) dW(s), \ \ t\in [0,T],
\end{equation}
where $\psi(\cdot)\in L_{\mathcal{F}_{T}}^2(0,T;\mathbb{R}^{m})$, and
$g:\Delta\times \mathbb{R}^{m}\times \mathbb{R}^{m\times d}\times \mathbb{R}^{m\times d}\times\Omega\longrightarrow \mathbb{R}^{m} $
is $\mathcal{B}(\Delta\times \mathbb{R}^{m}\times \mathbb{R}^{m\times d}\times \mathbb{R}^{m\times d})\otimes\mathcal{F}_{T}$-measurable such that $s\mapsto g(t,s,y,z,\vartheta)$ is $\mathcal{F}$-progressively measurable
for all $(t,y,z,\vartheta)\in [0,s]\times \mathbb{R}^{m}\times \mathbb{R}^{m\times d}\times \mathbb{R}^{m\times d}$, $s\in[0,T]$.

\begin{itemize}
\item[(H1)] Suppose there exists a constant $L>0$ such that,
$P$-a.s., for all $(t,s) \in \Delta, \ y,y' \in \mathbb{R}^{m}, \ z,z',\vartheta,$ $\vartheta' \in \mathbb{R}^{m\times d}$,
\begin{equation*}
\begin{split}
 & |g(t,s,y,z,\vartheta)-g(t,s,y',z',\vartheta')|\leq L\left(|y-y'| + \|z-z'\| + \|\vartheta-\vartheta'\|\right);\\
 & and \  E\int_0^T\int_t^T |g_{0}(t,s)|^{2} ds dt< \infty, \ where \ g_{0}(t,s)=g(t,s,0,0,0).
\end{split}
\end{equation*}
\end{itemize}

\begin{definition}
An adapted solution $(Y(\cdot),Z(\cdot,\cdot))$ of BSVIE (\ref{22}) is called an adapted M-solution if the following holds:
\begin{equation}\label{54}
  Y(t) =E[Y(t)] + \int_0^t Z(t,s)dW(s), \ \    0\leq t\leq T.
\end{equation}
\end{definition}

The following propositions can be found in \cite{Wang,Yong,Yong3}.

\begin{proposition}\label{19}
Under the assumption (H1), for any $\psi(\cdot)\in L_{\mathcal{F}_{T}}^2(0,T;\mathbb{R}^{m})$,  BSVIE (\ref{22}) admits a unique adapted M-solution.
\end{proposition}

\begin{proposition}\label{2}
Consider the following simple BSVIE
\begin{equation*}\label{8}
Y(t)=\psi(t) + \int_t^T g(t,s) ds - \int_t^T Z(t,s) dW(s), \ t\in [0,T],
\end{equation*}
 where $\psi(\cdot)\in L_{\mathcal{F}_{T}}^{2}(0,T;\mathbb{R}^{m})$ and  $g\in L_{\mathcal{F}}^2(\Delta;\mathbb{R}^{m})$.
Then
the above equation has a unique adapted solution $(Y,Z)\in \mathcal{H}_{\Delta}^{2}$, and the following estimate holds:
\begin{equation}\label{10}
\begin{split}
      &E  \int_0^T \left(e^{\beta t}|Y(t)|^{2} + \int_t^T e^{\beta s}\|Z(t,s)\|^{2} ds\right) dt\\
  \leq&  Ce^{\beta T}E \int_0^T |\psi(t)|^{2} dt + \frac{C}{\beta}E \int_0^T \int_t^T e^{\beta s}|g(t,s)|^{2} dsdt.
\end{split}
\end{equation}
Hereafter C is a positive constant which may be different from line to line.
\end{proposition}

\begin{proposition}\label{16}
For $i=0,1$, assume $g^{i}=g^{i}(t,s,y,z)$ satisfies (H1).
Let $(Y^{i},Z^{i})\in \mathcal{H}_{\Delta}^{2}$ be respectively the solutions of the following BSVIEs,
\begin{equation*}
 Y^{i}(t)=\psi^{i}(t) + \int_t^T g^{i}(t,s,Y^{i}(s),Z^{i}(t,s)) ds  - \int_t^T Z^{i}(t,s) dW(s), \ \ t\in [0,T].
\end{equation*}
Suppose $\overline{g}(t,s,y,z)$ satisfies (H1) such that
$y\mapsto \overline{g}(t,s,y,z)$ is nondecreasing with
   \begin{equation*}
     g^{0}(t,s,y,z)\leq \overline{g}(t,s,y,z)\leq g^{1}(t,s,y,z),
      \ \ \forall (t,y,z)\in  [0,s]\times \mathbb{R}^{m}\times \mathbb{R}^{m\times d}, \ a.s., \ a.e. \ s\in[0,T].
   \end{equation*}
Moreover, $\overline{g}_{z}(t,s,y,z)$ exists and
   \begin{equation*}
     \overline{g}_{z_{1}}(t,s,y,z),...,\overline{g}_{z_{d}}(t,s,y,z) \in \mathbb{R}_{d}^{m\times m},
      \ \ \forall (t,y,z)\in  [0,s]\times \mathbb{R}^{m}\times \mathbb{R}^{m\times d}, \ a.s., \ a.e. \ s\in[0,T].
   \end{equation*}
Then for any $\psi^{i}(\cdot)\in L_{\mathcal{F}_{T}}^2(0,T;\mathbb{R}^{m})$ satisfying
$ \psi^{0}(t)\leq \psi^{1}(t), \ \  a.s., \ t\in[0,T],$
the corresponding unique adapted solution $(Y^{i},Z^{i})\in \mathcal{H}_{\Delta}^{2}$ satisfies
   \begin{equation*}
   Y^{0}(t)\leq Y^{1}(t), \ \ a.s., \ t\in[0,T].
   \end{equation*}
\end{proposition}

\section{Existence and uniqueness theorem}

We now consider a new form of BSVIEs as follows:
\begin{equation}\label{4}
  \begin{cases}
  Y(t)=\psi(t) + \int_t^T g(t,s,Y(s),Z(t,s),Z(s,t),Y(s+\delta_{s}),Z(t,s+\zeta_{s}),Z(s+\zeta_{s},t)) ds \\
      \ \ \ \ \ \ \  \ \ \ - \int_t^T Z(t,s) dW(s), \ \ t\in [0,T]; \\
  Y(t)=\psi(t), \ \ t\in [T,T+K];\\
Z(t,s)=\eta(t,s), \ \ (t,s)\in [0,T+K]^{2}\setminus [0,T]^{2}. \\
  \end{cases}
\end{equation}
where $\delta_{\cdot}$ and $\zeta_{\cdot}$ are two $\mathbb{R}^{+}$-valued continuous functions defined on $[0,T]$ such that:

\begin{itemize}
\item[(i)] There exists a constant $K\geq 0$ such that, for all $ s\in [0,T]$,
  $$
  s+\delta_{s}\leq T+K; \ \ s+\zeta_{s}\leq T+K.
  $$

\item[(ii)] There exists a constant $M\geq 0$ such that, for all non-negative and integrable $g_{1}(\cdot), g_{2}(\cdot,\cdot)$, $t\in [0,T]$,
\begin{equation*}
\begin{cases}
  \int_t^T g_{1}(s+\delta_{s}) ds\leq M\int_t^{T+K} g_{1}(s) ds;   \\
   \int_t^T g_{2}(t,s+\zeta_{s}) ds\leq M \int_t^{T+K} g_{2}(t,s) ds; \\
   \int_t^T g_{2}(s+\zeta_{s},t) ds\leq M \int_t^{T+K} g_{2}(s,t) ds.
\end{cases}
\end{equation*}
\end{itemize}

We call equation (\ref{4}) the anticipating BSVIE.

%

   Assume that for all $(t,s)\in \Delta,$
   $g(t,s,y,z,x,\xi,\eta,\varsigma,\omega): \mathbb{R}^{m}\times \mathbb{R}^{m\times d} \times \mathbb{R}^{m\times d}
   \times L^{2}(\mathcal{F}_{r_{1}};\mathbb{R}^{m}) \times  L^{2}(\mathcal{F}_{r_{2}};\mathbb{R}^{m\times d})\times  L^{2}(\mathcal{F}_{r_{3}};\mathbb{R}^{m
   \times d})\times\Omega
   \longrightarrow L^{2}(\mathcal{F}_{s};\mathbb{R}^{m})$, where $r_{1},r_{2},r_{3} \in[s,T+K],$ and $g$ satisfies the following conditions:

\begin{itemize}
 \item [(H2)] There exists a constant $L>0$, such that, $P$-a.s.,  for all $(t,s)\in \Delta, y,y'\in \mathbb{R}^{m},z,z',x,x'\in {R}^{m\times d},
      \xi(\cdot),\xi'(\cdot)\in L_{\mathcal{F}}^2(s,T+K;\mathbb{R}^{m}),\eta(t,\cdot),\eta'(t,\cdot),
      \varsigma(\cdot,t),$ $\varsigma'(\cdot,t)\in L_{\mathcal{F}}^2(s,T+K;\mathbb{R}^{m\times d}), r,r'\in [s,T+K],$ we have
\begin{equation*}
 \begin{split}
      &|g(t,s,y,z,x,\xi(r),\eta(t,r'),\varsigma(r',t))-g(t,s,y',z',x',\xi'(r),\eta'(t,r'),\varsigma'(r',t)) | \\
  \leq& L\bigg(|y-y'|+\|z-z'\|+\|x-x'\| \\
      & +E^{\mathcal{F}_{s}}\left[|\xi(r)-\xi'(r)|+\|\eta(t,r')-\eta'(t,r')\|+\|\varsigma(r',t)-\varsigma'(r',t)\|\right]\bigg);\\
 &and \  E\int_0^T \int_t^T |g_{0}(t,s)|^{2} ds dt < \infty, \ where \ g_{0}(t,s)=g(t,s,0,0,0,0,0,0). \\
 \end{split}
\end{equation*}
\end{itemize}

\begin{remark}
Note that for all $(t,s)\in \Delta,$
$g(t,s,\cdot,\cdot,\cdot,\cdot,\cdot,\cdot)$ is $\mathcal{F}_{s}$-measurable ensures the solution to the anticipating BSVIE is $\mathcal{F}_{s}$-adapted.
\end{remark}

In order to establish the well-posedness of anticipating BSVIE (\ref{4}), we introduce the following space.
For any $\beta\geq 0$, we let $\mathcal{H}^{2}[0,T+K]$
be the space of all pairs
$$(Y,Z)\in L_{\mathcal{F}}^2(0,T+K;\mathbb{R}^{m})\times L_{\mathcal{F}}^2([0,T+K]^{2};\mathbb{R}^{m\times d})$$
 under the following norm
\begin{equation*}
  \| (Y(\cdot),Z(\cdot,\cdot)) \|_{\mathcal{H}^{2}[0,T+K]}
  \equiv \left[ E\int_0^{T+K} \left(e^{\beta t}|Y(t)|^{2} + \int_0^{T+K} e^{\beta s}\|Z(t,s)\|^{2} ds\right) dt \right]^{\frac{1}{2}}< \infty.
\end{equation*}
Different from $\mathcal{H}_{\Delta}^{2}$ defined in the previous section, we see that $Z(\cdot,\cdot)$ is defined on $[0,T+K]^{2}$.
Similar to $\mathcal{H}_{\Delta}^{2}$, we know $\mathcal{H}^{2}[0,T+K]$ is also a Hilbert space.

Next, let $\mathcal{M}^{2}[0,T+K]$ be the set of all pairs $(Y,Z)\in \mathcal{H}^{2}[0,T+K]$ such that Eq. (\ref{54}) holds in $[0,T+K]$, i.e.,
\begin{equation}\label{24}
  Y(t) =E[Y(t)] + \int_0^t Z(t,s)dW(s), \ \    0\leq t\leq T+K.
\end{equation}
Then for any  $(Y,Z)\in \mathcal{M}^{2}[0,T+K]$, one can show that
\begin{equation}\label{7}
\begin{split}
  &E\int_0^{T+K} \left(e^{\beta t}|Y(t)|^{2} + \int_0^{T+K} e^{\beta s}\|Z(t,s)\|^{2} ds \right) dt\\
 \leq& 2E\int_0^{T+K} \left(e^{\beta t}|Y(t)|^{2} + \int_t^{T+K} e^{\beta s}\|Z(t,s)\|^{2} ds \right) dt,
\end{split}
\end{equation}
since from (\ref{24}) one has
\begin{equation}\label{23}
  E\int_0^{T+K} \int_0^t e^{\beta s}\|Z(t,s)\|^{2} ds dt \leq E\int_0^{T+K} e^{\beta t}|Y(t)|^{2} dt.
\end{equation}
This means that we can use the following as an equivalent norm in $\mathcal{M}^{2}[0,T+K]$:
 \begin{equation*}
 \| (Y(\cdot),Z(\cdot,\cdot)) \|_{\mathcal{M}^{2}[0,T+K]}
 \equiv \bigg[ E\int_0^{T+K} \bigg(e^{\beta t}|Y(t)|^{2} + \int_t^{T+K} e^{\beta s}\|Z(t,s)\|^{2} ds\bigg) dt \bigg]^{\frac{1}{2}}.
 \end{equation*}
We also let $\overline{\mathcal{M}}^{2}[0,T+K]$ be the set of all pairs
$(\psi,\eta)\in L_{\mathcal{F}_{T}}^2(0,T+K;\mathbb{R}^{m})\times L_{\mathcal{F}}^2([0,T+K]^{2};\mathbb{R}^{m\times d})$
such that
\begin{equation*}
  \psi(t) =E[\psi(t)] + \int_0^t \eta(t,s)dW(s), \ \    0\leq t\leq T+K.
\end{equation*}
Since the space $L_{\mathcal{F}_{T}}^2(T,T+K;\mathbb{R}^{m})$ is equivalent to the space $L_{\mathcal{F}}^2(T,T+K;\mathbb{R}^{m})$,
the space $\overline{\mathcal{M}}^{2}[T,T+K]$ is equivalent to the space $\mathcal{M}^{2}[T,T+K]$ too.

We now state and proof the well-posedness theorem for anticipating BSVIE (\ref{4}).
\begin{theorem}\label{0}
  Suppose that $g$ satisfies (H2), and $\delta,\zeta$ satisfy (i) and (ii).
  Then for any $(\psi(\cdot),\eta(\cdot,\cdot))\in \overline{\mathcal{M}}^{2}[0,T+K]$,
  the anticipating BSVIE (\ref{4}) admits a unique adapted M-solution $(Y(\cdot),Z(\cdot,\cdot))\in \mathcal{H}^{2}[0,T+K]$.
\end{theorem}

\begin{proof}
For any  $(y(\cdot),z(\cdot,\cdot))\in \mathcal{M}^{2}[0,T+K]$, consider the following BSVIE:
\begin{equation}\label{6}
  \begin{cases}
Y(t)=\psi(t) + \int_t^T \overline{g}(t,s) ds - \int_t^T Z(t,s) dW(s), \ \ t\in [0,T]; \\
Y(t)=\psi(t), \ \ t\in [T,T+K];\\
Z(t,s)=\eta(t,s), \ \ (t,s)\in [0,T+K]^{2}\setminus [0,T]^{2}, \\
  \end{cases}
\end{equation}
where
$$\overline{g}(t,s)=g(t,s,y(s),z(t,s),z(s,t),y(s+\delta_{s}),z(t,s+\zeta_{s}),z(s+\zeta_{s},t)).$$
From Proposition \ref{19} and \ref{2}, Eq. (\ref{6}) admits a unique adapted solution
$(Y(\cdot),Z(\cdot,\cdot))\in \mathcal{H}_{\widetilde{\Delta}}^{2}$.
Now we define $Z(\cdot,\cdot)$ on $\widetilde{\Delta}^{c}$ from the following:
\begin{equation*}
  Y(t)=EY(t) + \int_0^t Z(t,s) dW(s), \ \ t\in [0,T+K].
\end{equation*}
Then $(Y(\cdot),Z(\cdot,\cdot))\in \mathcal{M}^{2}[0,T+K]$ is an adapted M-solution to Eq. (\ref{6}).
Thus, the maping $(y(\cdot),z(\cdot,\cdot))\mapsto (Y(\cdot),Z(\cdot,\cdot))$ is well-defined.
By the estimate (\ref{10}) in Proposition \ref{2}, one has
\begin{equation*}
 \begin{split}
        &E\int_0^T \left(e^{\beta t}|Y(t)|^{2} + \int_t^T e^{\beta s}\|Z(t,s)\|^{2} ds\right) dt \\
     \leq& Ce^{\beta T}E \int_0^T |\psi(t)|^{2} dt
          + \frac{C}{\beta}E\int_0^T \int_t^T e^{\beta s}|\overline{g}(t,s)|^{2} dsdt.\\
 \end{split}
\end{equation*}
From (H2) and (\ref{23}), and note that $\delta,\zeta$ satisfy (i) and (ii), we have
\begin{equation}\label{25}
  \begin{split}
      &E\int_0^T \int_t^T e^{\beta s}|\overline{g}(t,s)|^{2} dsdt\\
  \leq& 5L^{2}E \int_0^T \int_t^T e^{\beta s}\bigg(|g_{0}(t,s)|^{2} + |y(s)|^{2} + \|z(t,s)\|^{2} + \|z(s,t)\|^{2} \\
      & \ \ \ \ \ \ \ \ \ \ \ \ \ \ \ \ \ \ \ \ \ \ \ \ +3\big[|y(s+\delta_{s})|^{2} + \|z(t,s+\zeta_{s})\|^{2} + \|z(s+\zeta_{s},t)\|^{2}\big]\bigg) dsdt \\
  \leq& 5L^{2}E \int_0^T \int_t^T e^{\beta s}|g_{0}(t,s)|^{2} dsdt + 10L^{2}(T+1)E \int_0^T \left(e^{\beta t}|y(t)|^{2}
        + \int_0^T e^{\beta s}\|z(t,s)\|^{2} ds\right)dt \\
      &+30L^{2}M(T+K+1)E \int_0^{T+K}\left(e^{\beta t}|y(t)|^{2} + \int_0^{T+K} e^{\beta s}\|z(t,s)\|^{2} ds\right)dt\\
  \leq& 5L^{2}E \int_0^T \int_t^T e^{\beta s}|g_{0}(t,s)|^{2} dsdt \\
      & +60L^{2}(M+1)(T+K+1)E \int_0^{T+K}\left(e^{\beta t}|y(t)|^{2} + \int_0^{T+K} e^{\beta s}\|z(t,s)\|^{2} ds\right)dt.\\
   \end{split}
\end{equation}
Hence
\begin{equation}\label{26}
  \begin{split}
      &E  \int_0^T \left(e^{\beta t}|Y(t)|^{2} + \int_t^T e^{\beta s}\|Z(t,s)\|^{2} ds\right) dt\\
  \leq&  Ce^{\beta T}E \int_0^T |\psi(t)|^{2} dt + \frac{C}{\beta}E \int_0^T \int_t^T e^{\beta s}|g_{0}(t,s)|^{2} dsdt \\
      &+ \frac{C}{\beta}E \int_0^{T+K}\left(e^{\beta t}|y(t)|^{2} + \int_0^{T+K} e^{\beta s}\|z(t,s)\|^{2} ds\right)dt.\\
   \end{split}
\end{equation}
Now if $(Y_{i}(\cdot),Z_{i}(\cdot,\cdot))$ is the corresponding adapted M-solution of $(y_{i}(\cdot),z_{i}(\cdot,\cdot))$ to BSVIE (\ref{6}), $i=1,2$, note (\ref{7}), then
\begin{equation*}
  \begin{split}
     &E  \int_0^{T+K} \left(e^{\beta t}|Y_{1}(t)-Y_{2}(t)|^{2} + \int_0^{T+K} e^{\beta s}\|Z_{1}(t,s)-Z_{2}(t,s)\|^{2} ds\right) dt\\
  \leq& \frac{C}{\beta}E\int_0^{T+K} \left(e^{\beta t}|y_{1}(t)-y_{2}(t)|^{2} + \int_0^{T+K} e^{\beta s}\|z_{1}(t,s)-z_{2}(t,s)\|^{2} ds\right) dt.
   \end{split}
\end{equation*}
Let $\beta=2C+1$, then $(y(\cdot),z(\cdot,\cdot))\mapsto (Y(\cdot),Z(\cdot,\cdot))$ is a contraction on $\mathcal{M}^{2}[0,T+K]$. This completes the proof.
\end{proof}

\begin{proposition}
Under the assumptions of Theorem \ref{0}, the solution of anticipating BSVIE (\ref{4}) satisfies
\begin{equation}\label{28}
\begin{split}
       &  E\bigg[\int_0^T e^{\beta t} |Y(t)|^{2} dt + \int_0^T \int_t^T e^{\beta s} \|Z(t,s)\|^{2} dsdt\bigg]\\
   \leq& CE\bigg[\int_0^{T+K} |\psi(t)|^{2} dt +\int_0^T \int_t^T e^{\beta s}|g_{0}(t,s)|^{2} dsdt
        + \int_T^{T+K} \int_T^{T+K} e^{\beta s}\|\eta(t,s)\|^{2} dsdt\\
       & + \int_0^T \int_T^{T+K} \big(e^{\beta s}\|\eta(t,s)\|^{2}+ e^{\beta t}\|\eta(s,t)\|^{2}\big) dsdt \bigg].
\end{split}
\end{equation}
\end{proposition}

\begin{proof}[Sketch proof]
By the estimate (\ref{10}), similar to (\ref{25}) and (\ref{26}), and note (\ref{7}), we have
\begin{equation*}
  \begin{split}
      &E  \left(\int_0^T e^{\beta t}|Y(t)|^{2}dt + \int_0^T\int_0^T e^{\beta s}\|Z(t,s)\|^{2} ds dt\right)\\
  \leq&  Ce^{\beta T}E \int_0^T |\psi(t)|^{2} dt + \frac{C}{\beta}E \int_0^T \int_t^T e^{\beta s}|g_{0}(t,s)|^{2} dsdt \\
      &+ \frac{C}{\beta}E \int_0^{T+K}e^{\beta t}|Y(t)|^{2}dt + \frac{C}{\beta}E \int_0^{T+K}\int_0^{T+K} e^{\beta s}\|Z(t,s)\|^{2} dsdt.
   \end{split}
\end{equation*}
Since
\begin{equation*}
  \begin{split}
      & \frac{C}{\beta}E \int_0^{T+K}e^{\beta t}|Y(t)|^{2}dt + \frac{C}{\beta}E \int_0^{T+K}\int_0^{T+K} e^{\beta s}\|Z(t,s)\|^{2} dsdt\\
     =& \frac{C}{\beta}E \bigg(\int_0^{T}e^{\beta t}|Y(t)|^{2}dt + \int_0^{T}\int_0^{T} e^{\beta s}\|Z(t,s)\|^{2} dsdt\bigg)\\
      &+ \frac{C}{\beta}E\int_T^{T+K} e^{\beta t}|\psi(t)|^{2}dt + \frac{C}{\beta}E\int_T^{T+K} \int_T^{T+K} e^{\beta s}\|\eta(t,s)\|^{2} dsdt\\
      &+ \frac{C}{\beta}E \bigg(\int_0^{T}\int_T^{T+K} e^{\beta s}\|\eta(t,s)\|^{2} dsdt + \int_T^{T+K}\int_0^{T} e^{\beta s}\|\eta(t,s)\|^{2} dsdt \bigg).
   \end{split}
\end{equation*}
Now let $\beta=2C$, then we obtain the estimate (\ref{28}).

\end{proof}

\section{Comparison theorem}

In this section we prove a comparison theorem for ABSVIEs of the following type: For $i=0,1,$
\begin{equation}\label{9}
  \begin{cases}
  Y^{i}(t)=\psi^{i}(t) + \int_t^T g^{i}(t,s,Y^{i}(s),Z^{i}(t,s),Y^{i}(s+\delta_{s})) ds - \int_t^T Z^{i}(t,s) dW(s), \ \ t\in [0,T]; \\
  Y^{i}(t)=\psi^{i}(t), \ \ t\in [T,T+K].
  \end{cases}
\end{equation}
For ABSVIEs of the above form, we need only the values $Z^{i}(t,s)$ of $Z^{i}(\cdot,\cdot)$
for $(t,s)\in \Delta$ and the notation of $M$-solution is not necessary.
%
%
%

It is easy to see that under the assumption of Theorem \ref{0}, for any $\psi^{i}(\cdot)\in L_{\mathcal{F}_{T}}^2(0,T+K;\mathbb{R}^{m})$,
ABSVIE (\ref{9}) admits a unique adapted solution
$(Y^{i}(\cdot),Z^{i}(\cdot,\cdot))\in L_{\mathcal{F}}^2(0,T+K;\mathbb{R}^{m})\times L_{\mathcal{F}}^2(\Delta;\mathbb{R}^{m\times d})$.

\begin{theorem}\label{15}
Let $\delta$ and $\zeta$ satisfy (i)-(ii), and $g^{i}$ satisfies (H2), $i=1,2$.
Suppose $\overline{g}=\overline{g}(t,s,y,z,\xi)$ satisfies (H2) and for all $(t,s,y,z)\in \Delta\times\mathbb{R}^{m}\times \mathbb{R}^{m\times d}$,
$\overline{g}(t,s,y,z,\cdot)$ is increasing, i.e.,
$\overline{g}(t,s,y,z,\xi_{1}(r))\leq \overline{g}(t,s,y,z,\xi_{2}(r))$,
if $\xi_{1}(r)\leq \xi_{2}(r)$, $\xi_{1}(\cdot),\xi_{2}(\cdot)\in L^{2}_{\mathcal{F}}(s,T+K;\mathbb{R})$, $r\in [s,T+K]$.
Moreover
\begin{equation*}
 \begin{split}
     g^{0}&(t,s,y,z,\xi)\leq \overline{g}(t,s,y,z,\xi)\leq g^{1}(t,s,y,z,\xi), \\
      \ \ &\forall (t,s,y,z,\xi)\in  \Delta\times \mathbb{R}^{m}\times \mathbb{R}^{m\times d}\times L^{2}(\mathcal{F}_{r};\mathbb{R}^{m}), \ a.s., \ a.e.,
\end{split}
\end{equation*}
and $\overline{g}_{z}(t,s,y,z,\xi)$ exists with
\begin{equation*}
 \begin{split}
      \overline{g}_{z_{1}}&(t,s,y,z,\xi),...,\overline{g}_{z_{d}}(t,s,y,z,\xi) \in \mathbb{R}_{d}^{m\times m},\\
      \ \ &\forall (t,s,y,z,\xi)\in  \Delta\times \mathbb{R}^{m}\times \mathbb{R}^{m\times d}\times L^{2}(\mathcal{F}_{r};\mathbb{R}^{m}), \ a.s., \ a.e.
 \end{split}
\end{equation*}
Then for any $\psi^{i}(\cdot)\in L_{\mathcal{F}_{T}}^2(0,T+K;\mathbb{R}^{m})$ satisfying $\psi^{0}(t)\leq \psi^{1}(t),  \ a.s., \ t\in[0,T+K],$
we have
   \begin{equation*}
   Y^{0}(t)\leq Y^{1}(t), \ a.s., \ t\in[0,T+K].
   \end{equation*}
\end{theorem}

\begin{proof}
Let $\overline{\psi}(\cdot)\in L_{\mathcal{F}_{T}}^2(0,T+K;\mathbb{R}^{m})$ and
\begin{equation*}
  \psi^{0}(t)\leq \overline{\psi}(t)\leq \psi^{1}(t),  \ a.s. \ t\in[0,T+K].
\end{equation*}
Let $(\overline{Y}(\cdot),\overline{Z}(\cdot,\cdot))\in L_{\mathcal{F}}^2(0,T+K;\mathbb{R}^{m})\times L_{\mathcal{F}}^2(\Delta;\mathbb{R}^{m\times d})$
 be the unique adapted solution to the following ABSVIE:
\begin{equation}\label{11}
  \begin{cases}
  \overline{Y}(t)=\overline{\psi}(t) + \int_t^T \overline{g}(t,s,\overline{Y}(s),\overline{Z}(t,s),
  \overline{Y}(s+\delta_{s}) ds    - \int_t^T \overline{Z}(t,s) dW(s), \ \  t\in [0,T]; \\
  \overline{Y}(t)=\overline{\psi}(t), \ \   t\in [T,T+K].\\
  \end{cases}
\end{equation}
Now we set $\widetilde{Y}_{0}(\cdot)=Y^{1}(\cdot)$ and consider the following BSVIE:
\begin{equation*}
  \begin{cases}
  \widetilde{Y}_{1}(t)=\overline{\psi}(t) + \int_t^T \overline{g}(t,s,\widetilde{Y}_{1}(s),\widetilde{Z}_{1}(t,s),
     \widetilde{Y}_{0}(s+\delta_{s})) ds  - \int_t^T \widetilde{Z}_{1}(t,s) dW(s), \ \ t\in [0,T]; \\
  \widetilde{Y}_{1}(t)=\overline{\psi}(t),  \ \  t\in [T,T+K].\\
  \end{cases}
\end{equation*}
Let $(\widetilde{Y}_{1}(\cdot),\widetilde{Z}_{1}(\cdot,\cdot))\in L_{\mathcal{F}}^2(0,T+K;\mathbb{R}^{m})\times L_{\mathcal{F}}^2(\Delta;\mathbb{R}^{m\times d})$
be the unique adapted solution to the above equation. Since
\begin{equation*}
  \begin{cases}
   \overline{g}(t,s,y,z,\widetilde{Y}_{0}(s+\delta_{s}))
   \leq g^{1}(t,s,y,z,\widetilde{Y}_{0}(s+\delta_{s})), \ \
       (t,s,y,z)\in \Delta \times \mathbb{R}^{m}\times \mathbb{R}^{m\times d}, \ a.s., \ a.e.;\\
   \overline{\psi}(t) \leq \psi^{1}(t),\  \ \ a.s. \ \ t\in[0,T+K].
  \end{cases}
\end{equation*}
By Proposition \ref{16}, we obtain that
\begin{equation*}
  \widetilde{Y}_{1}(t)\leq \widetilde{Y}_{0}(t), \ \ a.s. \ t\in[0,T+K].
\end{equation*}
Next, we consider the following BSVIE:
\begin{equation*}
  \begin{cases}
  \widetilde{Y}_{2}(t)=\overline{\psi}(t) + \int_t^T \overline{g}(t,s,\widetilde{Y}_{2}(s),\widetilde{Z}_{2}(t,s),\widetilde{Y}_{1}(s+\delta_{s})) ds
    - \int_t^T \widetilde{Z}_{2}(t,s) dW(s), \ \ t\in [0,T]; \\
  \widetilde{Y}_{2}(t)=\overline{\psi}(t), \ \ t\in [T,T+K].\\
  \end{cases}
\end{equation*}
Let $(\widetilde{Y}_{2}(\cdot),\widetilde{Z}_{2}(\cdot,\cdot))
\in L_{\mathcal{F}}^2(0,T+K;\mathbb{R}^{m})\times L_{\mathcal{F}}^2(\Delta;\mathbb{R}^{m\times d})$
 be the adapted solution to the above equation. Now, since $\xi\mapsto \overline{g}(t,s,y,z,\xi)$ is increasing, we have
\begin{equation*}
   \overline{g}(t,s,y,z,\widetilde{Y}_{1}(s+\delta_{s})) \leq \overline{g}(t,s,y,z,\widetilde{Y}_{0}(s+\delta_{s})), \ \
    (t,s,y,z)\in \Delta\times \mathbb{R}^{m}\times \mathbb{R}^{m\times d}, \ a.s., \ a.e.
\end{equation*}
Hence, similar to the above, we obtain
\begin{equation*}
  \widetilde{Y}_{2}(t)\leq \widetilde{Y}_{1}(t), \ \ \ a.s. \ t\in[0,T+K].
\end{equation*}
By induction, we can construct a sequence
$\{(\widetilde{Y}_{k}(\cdot),\widetilde{Z}_{k}(\cdot,\cdot))\}_{k\geq 1}
\in L_{\mathcal{F}}^2(0,T+K;\mathbb{R}^{m})\times L_{\mathcal{F}}^2(\Delta;\mathbb{R}^{m\times d})$
such that
\begin{equation*}
  \begin{cases}
  \widetilde{Y}_{k}(t)=\overline{\psi}(t) + \int_t^T \overline{g}(t,s,\widetilde{Y}_{k}(s),\widetilde{Z}_{k}(t,s),
  \widetilde{Y}_{k-1}(s+\delta_{s})) ds - \int_t^T \widetilde{Z}_{k}(t,s) dW(s), \ t\in [0,T]; \\
  \widetilde{Y}_{k}(t)=\overline{\psi}(t), \ t\in [T,T+K].\\
  \end{cases}
\end{equation*}
Similarly, we deduce
\begin{equation*}
  Y^{1}(t)=\widetilde{Y}_{0}(t)\geq \widetilde{Y}_{1}(t)\geq \widetilde{Y}_{2}(t)\cdots, \ \ \ a.s. \ t\in[0,T+K].
\end{equation*}
Next we will show that the sequence $\{(\widetilde{Y}_{k}(\cdot),\widetilde{Z}_{k}(\cdot,\cdot))\}_{k\geq 2}$  is Cauchy sequence.
By utilizing the estimate (\ref{10}), we have
\begin{equation*}
   \begin{split}
     &E\int_0^T \left(e^{\beta t}|\widetilde{Y}_{k}(t)-\widetilde{Y}_{k-1}(t)|^{2}
      + \int_t^T e^{\beta s}\|\widetilde{Z}_{k}(t,s)-\widetilde{Z}_{k-1}(t,s)\|^{2} ds \right) dt\\
\leq& \frac{C}{\beta}E\int_0^T \int_t^T e^{\beta s}
      \bigg(\overline{g}(t,s,\widetilde{Y}_{k}(s),\widetilde{Z}_{k}(t,s),\widetilde{Y}_{k-1}(s+\delta_{s}))\\
     & \ \ \ \ \ \ \ \ \ \ \ \ \ \ \ \ \ \ \ \ \ \ \
     -\overline{g}(t,s,\widetilde{Y}_{k-1}(s),\widetilde{Z}_{k-1}(t,s),\widetilde{Y}_{k-2}(s+\delta_{s}))\bigg)^{2} dsdt\\
\leq& \frac{C}{\beta}E\int_0^T \int_t^T e^{\beta s}\bigg(|\widetilde{Y}_{k}(s)-\widetilde{Y}_{k-1}(s)|^{2} +\|\widetilde{Z}_{k}(t,s)-\widetilde{Z}_{k-1}(t,s)\|^{2}\\
     & \ \ \ \ \ \ \ \ \ \ \ \ \ \ \ \ \ \ \ \ \ \ \ \  +|\widetilde{Y}_{k-1}(s+\delta_{s})-\widetilde{Y}_{k-2}(s+\delta_{s})|^{2}\bigg) dsdt\\
\leq& \frac{C}{\beta}E\int_0^T\left( e^{\beta t}|\widetilde{Y}_{k}(t)-\widetilde{Y}_{k-1}(t)|^{2}
      +\int_t^T e^{\beta s}\|\widetilde{Z}_{k}(t,s)-\widetilde{Z}_{k-1}(t,s)\|^{2} ds\right) dt \\
     & \ \ \ +\frac{C}{\beta}E\int_0^{T+K} e^{\beta t}|\widetilde{Y}_{k-1}(t)-\widetilde{Y}_{k-2}(t)|^{2} dt.
   \end{split}
\end{equation*}
Hence
\[\begin{split}\label{13}
   &(1-\frac{C}{\beta})E\int_0^T \left(e^{\beta t}|\widetilde{Y}_{k}(t)-\widetilde{Y}_{k-1}(t)|^{2}
      + \int_t^T e^{\beta s}\|\widetilde{Z}_{k}(t,s)-\widetilde{Z}_{k-1}(t,s)\|^{2} ds\right) dt \\
  \leq& \frac{C}{\beta} E\int_0^{T+K} e^{\beta t}|\widetilde{Y}_{k-1}(t)-\widetilde{Y}_{k-2}(t)|^{2} dt.
\end{split}\]
Note that the constant $C>0$ in the above can be chosen independent of $\beta\geq 0$.
Thus by choosing $\beta=3C$, we obtain
\[\begin{split}
    & E\int_0^{T+K} \left(e^{\beta t}|\widetilde{Y}_{k}(t)-\widetilde{Y}_{k-1}(t)|^{2}
       + \int_t^{T+K} e^{\beta s}\|\widetilde{Z}_{k}(t,s)-\widetilde{Z}_{k-1}(t,s)\|^{2} ds\right) dt\\
\leq& \frac{1}{2} E\int_0^{T+K} e^{\beta t}|\widetilde{Y}_{k-1}(t)-\widetilde{Y}_{k-2}(t)|^{2} dt\\
\leq& \frac{1}{2} E\int_0^{T+K} \left(e^{\beta t}|\widetilde{Y}_{k-1}(t)-\widetilde{Y}_{k-2}(t)|^{2}
      + \int_t^{T+K} e^{\beta s}\|\widetilde{Z}_{k-1}(t,s)-\widetilde{Z}_{k-2}(t,s)\|^{2} ds\right) dt\\
\leq& (\frac{1}{2})^{k-2} E\int_0^{T+K} \left(e^{\beta t}|\widetilde{Y}_{2}(t)-\widetilde{Y}_{1}(t)|^{2}
      + \int_t^{T+K} e^{\beta s}\|\widetilde{Z}_{2}(t,s)-\widetilde{Z}_{1}(t,s)\|^{2} ds \right) dt.
\end{split}\]
It follows that  $\{(\widetilde{Y}_{k}(\cdot),\widetilde{Z}_{k}(\cdot,\cdot))\}_{k\geq 1}$  is a Cauchy sequence in the Banach space
$ L_{\mathcal{F}}^2(0,T+K;\mathbb{R}^{m})\times L_{\mathcal{F}}^2(\Delta;\mathbb{R}^{m\times d})$.
Denote their limits by $\widetilde{Y}(\cdot)$ and $\widetilde{Z}(\cdot,\cdot)$, respectively.
Then $(\widetilde{Y}(\cdot),\widetilde{Z}(\cdot,\cdot))$ $\in  L_{\mathcal{F}}^2(0,T+K;\mathbb{R}^{m})\times L_{\mathcal{F}}^2(\Delta;\mathbb{R}^{m\times d})$ and
\begin{equation*}
  \lim_{k\rightarrow \infty}\left(E\int_0^{T+K} e^{\beta t}|\widetilde{Y}_{k}(t)-\widetilde{Y}(t)|^{2} dt
      + E\int_0^{T+K}\int_t^{T+K} e^{\beta s}\|\widetilde{Z}_{k}(t,s)-\widetilde{Z}(t,s)\|^{2} ds dt\right)=0,
\end{equation*}
also we have
\begin{equation*}
  \begin{cases}
  \widetilde{Y}(t)=\overline{\psi}(t) + \int_t^T \overline{g}(t,s,\widetilde{Y}(s),\widetilde{Z}(t,s),
  \widetilde{Y}(s+\delta_{s})) ds
   - \int_t^T \widetilde{Z}(t,s) dW(s), \ t\in [0,T]; \\
  \widetilde{Y}(t)=\overline{\psi}(t), \ t\in [T,T+K].
  \end{cases}
\end{equation*}
Connecting the above equation with the equation (\ref{11}), by Theorem \ref{0}, we have
\begin{equation*}
  \overline{Y}(t)= \widetilde{Y}(t), \ \ \ a.s. \ t\in[0,T+K].
\end{equation*}
Hence we obtain
\begin{equation*}
  \overline{Y}(t)\leq Y^{1}(t), \ \ \ a.s. \ t\in[0,T+K].
\end{equation*}
Similarly, we can prove that
\begin{equation*}
 Y^{0}(t)\leq \overline{Y}(t),  \ \ \ a.s. \ t\in[0,T+K].
\end{equation*}
Therefore, our conclusion follows.
\end{proof}

\begin{example}
Let $g^{0}(t,s,\xi(r))=-E^{\mathcal{F}_{s}}[|\xi(r)|]-ln2$, $g^{1}(t,s,\xi(r))=E^{\mathcal{F}_{s}}[|\xi(r)|]+\pi$.
    We choose  $\overline{g}(t,s,\xi(r))=E^{\mathcal{F}_{s}}[\xi(r)]+1$.
    It's easy to check that $g^{0},g^{1}$ and $\overline{g}$ satisfy the assumptions of Theorem \ref{15}.
    If the terminal condition satisfies $\psi^{0}(t)\leq \psi^{1}(t),  \ a.s., \ t\in[0,T+K],$
     we  derive
    $$ Y^{0}(t)\leq Y^{1}(t), \ \ a.s. \ t\in[0,T+K].$$
\end{example}

\section*{Acknowledgements}


\begin{thebibliography}{99}

\bibitem{Nualart} E. Al\`{o}s, D. Nualart, Anticipating stochastic Volterra equations, Stochastic Processe. Appl. 72 (1997) 73-95.

\bibitem{Berger}  M. Berger, V. Mizel, Volterra equation with It\^{o} integrals, I, II, J. Intergal Equations 2 (1980) 187-245, 319-337.

\bibitem{Wu}      L. Chen, Z. Wu, Maximum principle for the stochastic optimal control problem with delay and application, Automatica 46 (2010) 1074-1080.

\bibitem{Kromer}  K. Eduard, O. Ludger, Classical differentiability of BSVIEs and dynamic capital allocations,
                     (2014) Available at SSRN: http://ssrn.com/abstract=2379500

\bibitem{Peng2}   N. El Karoui, S. Peng, M.C. Quenez, Backward stochastic differential equations in finance, Math. Finance 7 (1997) 1-71.

\bibitem{Lin}     J. Lin, Adapted solution of a backward stochastic nonlinear Volterra integral equations, Stoch. Anal. Appl. 20 (2002) 165-183.

\bibitem{Ma}      J. Ma, J. Yong, Forward-Backward Stochastic Differential Equations and Their Applications, Springer-Verlag, Berlin, 1999.

\bibitem{Protter} E. Pardoux, P. Protter, Stochastic Volterra equations with anticipating coefficients, Ann. Probab. 18 (1990) 1635-1655.

\bibitem{Peng}    E. Pardoux, S. Peng, Adapted solution of a backward stochastic differential equation, Systems Control Lett. 4 (1990) 55-61.

\bibitem{Yang}    S. Peng, Z. Yang, Anticipated backward stochastic differential euquations, Ann. Probab. 37 (2009) 877-902.

\bibitem{Wen}     Y. Shi, J. Wen, J. Yong, Backward doubly stochastic Volterra integral equations and applications, Priprint.

\bibitem{Wang}    Y. Shi, T. Wang, Solvability of general backward stochastic Volterra integral equations, J. Korean Math. Soc. 49 (2012) 1301-1321.

\bibitem{Shi2}    Y. Shi, T. Wang, J. Yong, Mean-field backward stochastic Volterra integral equations, Dis. Cont. Dyn. System Ser. B 18 (2013) 1929-1967.

\bibitem{Shi4}    Y. Shi, T. Wang, J. Yong, Optimal Control Problems of Forward-Backward Stochastic Volterra Integral Equations,
                     Math. Control Related Fields, 5 (2015) 613-649.

\bibitem{Yong}    T. Wang, J. Yong, Comparison theorems for some backward stochastic Volterra integral equations, Stochastic Processe. Appl. 125 (2015) 1756-1798.

\bibitem{Yong2}   J. Yong, Backward stochastic Volterra integral equations and some related problems, Stochastic Process. Appl. 116 (2006) 779-795.

\bibitem{Yong4}   J. Yong, Continuous-time dynamic risk measures by backward stochastic Volterra integral equations, Appl. Anal. 86 (2007) 1429-1442.

\bibitem{Yong3}   J. Yong, Well-posedness and regularity of backward stochastic Volterra integral equations, Probab. Theory Related Fields 142 (2008) 21-77.

\bibitem{Yang2}   Z. Yang, R.J. Elliott, A converse comparison theorem for anticipated BSDEs and related non-linear expectations,
                     Stochastic Process. Appl. 123 (2013) 275-299.

\bibitem{Zhangx}  X. Zhang, Stochastic Volterra equations in Banach spaces and stochastic partial differential equation, J. Funct. Anal. 258 (2010) 1361-1425.

\end{thebibliography}
\end{document}